\documentclass[11pt,a4paper]{amsart}
\usepackage[marginparwidth=0pt,margin=20truemm]{geometry} % margin

\usepackage{mypackage} % My packages.
\usepackage{mycommand} % My commands.

\usepackage{amscd}
\usepackage{url}
\usepackage{physics} % \dv{z} でd/dzを出すなどができるようになる

\DeclareMathOperator{\modulo}{mod}
\renewcommand{\mod}{\modulo}

\newcommand{\pe}{\mathfrak{p}}
\newcommand{\Ypre}{Y^{\mathrm{pre}}}
\newcommand{\Xpre}{X^{\mathrm{pre}}}

\usepackage{ctable} 
\newenvironment{parts}[0]{%
    \noindent
    \begin{enumerate}[leftmargin=2em]%
   }%
  {\end{enumerate}}

\numberwithin{equation}{section} % Setting of equation numbers 
\usepackage{mytheoremeng} % My theorem environments (English) with cleveref package

\usepackage{tikz}
\usetikzlibrary{arrows.meta, decorations, decorations.markings}
\begin{document}
% --------------------------------------------------------------------------

\title[Rational iterated preimages under unicritical polynomials]{The number of rational iterated preimages of the origin under unicritical polynomial maps}

\author[K. Sano]{Kaoru Sano}
\address{NTT Institute for Fundamental Mathematics, NTT Communication Science Laboratories, NTT, Inc., 2-4, Hikaridai, Seika-cho, Soraku-gun, Kyoto 619-0237, Japan}
\email{kaoru.sano@ntt.com}

\date{\today}

% --------------------------------------------------------------------------

\begin{abstract}
We study rational iterated preimages of the origin under unicritical maps $f_{d,c}(x)=x^d+c$. Earlier works of Faber--Hutz--Stoll and Hutz--Hyde--Krause established finiteness and conditional bounds in the quadratic case. Building on this, we prove that for $d=2$ and $c \in \mathbb Q\setminus\{0,-1\}$ there are no rational fourth preimages of the origin, and for all $d \geq 3$ there are no rational second preimages outside trivial cases.
The proof relies on geometric analysis of preimage curves, the elliptic Chabauty method, and Diophantine reduction.
As a result, we determine the number of rational iterated preimages of $0$ under $f_{d,c}$ for all $d\geq 2$.
\end{abstract}

\maketitle

% --------------------------------------------------------------------------

\tableofcontents

% --------------------------------------------------------------------------

% --------------------------------------------------------------------------

\section{Introduction} \label{sec: intro}

% --------------------------------------------------------------------------
Let $L$ be a number field.
Fix an integer $d\geq 2$ and $c \in L$, and define an endomorphism $f_{d,c}$ of affine line by
\[
    f_{d,c} \colon \bbA^1 \longrightarrow \bbA^1;\quad x \mapsto x^d +c.
\]
For $a\in L$, the rational iterated preimage of $a$ is the set
\[
    \bigcup_{N\geq 1} f_{d,c}^{-N}(a)(L) = \{x\in \bbA^1(L)\mid f_{d,c}^{\circ N}(x) = a \text{ for some }N\geq 1\},
\]
where $f_{d,c}^{\circ N}$ is the $N$-th iterate of $f_{d,c}$, and $f_{d,c}^{-N}(a)$ is the preimage of $a$ under $f_{d,c}^{\circ N}$.
Finding an $L$-rational iterated preimage of $a$ under $f_{d,c}$ amounts to solving complicated polynomial equations.
When we vary $c$, the problem is reduced to determining $L$-rational points $(c,x)$ of the preimage curves $\Ypre(d, N, a)$ defined by $f_{d,c}^{\circ N}(x) - a = 0$, or its normal projective model $\Xpre(d, N, a)$.
To study the uniformity of the number of rational points of such curves, define the quantities
\begin{align*}
    \kappa(d,a, L) &\coloneqq \sup_{c\in L}\left\{ \#\bigcup_{N\geq 1}f_{d,c}^{-N}(a)(L)\right\}\\
    \overline{\kappa}(d,a,L) &\coloneqq \limsup_{c\in L}\left\{ \#\bigcup_{N\geq 1}f_{d,c}^{-N}(a)(L)\right\}
\end{align*}
for an integer $d \geq 2$ and a point $a\in \bbA^1(L)$.
In \cite{FHIJMTZ09} and \cite{HHK11}, they deal with the uniformity of the number of rational iterated preimages under quadratic polynomial maps $f_{2,c}$.
\begin{thm}
    \begin{parts}
        \item \cite[Theorem 1.2 for $B=1$]{FHIJMTZ09}
            For all $a\in \bbA^1(L)$,
            the quantity $\kappa(2, a, L)$ is finite.
        \item \cite[Theorem 1.3]{HHK11}
            There is a finite set $S\subset \overline{\Q}$ such that we have
        \[
            \overline{\kappa}(2,a,L) =
            \begin{cases}
                10  & \text{if } a = -1/4,\\
                6   & \text{if } a \text{ is one of the three third critical values},\\
                4   & \text{if } a \in S\cap L,\\
                6   & \text{otherwise}.
            \end{cases}
        \]
    \end{parts}
\end{thm}
The explicit value $\kappa(2,0,\Q)$ is conditionally given in \cite{FHS11}.
\begin{thm}[{\cite[Theorem 1.3]{FHS11}}]\label{thm: FHS}
    Suppose that for $c\in \Q\setminus\{0,-1\}$ the morphism $f_{2,c}$ admits no rational fourth preimages of $0$.
    Then $\kappa(2,0,\Q)= 6$.
\end{thm}
Moreover, they proved that the morphism $f_{2,c}$ admits no rational fourth preimages of $0$ for $c\in \Q\setminus \{0,-1\}$ under the Birch--Swinnerton-Dyer conjecture for the Jacobian of $\Xpre(2, 4, 0)$.
This paper aims to remove the assumptions of \cref{thm: FHS} and extend it to unicritical polynomial maps of high degrees.

\begin{thm}\label{thm: main}
    The following statements hold.
    \begin{parts}
        \item \label{item: thm: d = 2}
            For a rational number $c\in \Q\setminus\{0,-1\}$, the morphism $f_{2,c}$ admits no rational fourth preimages of $0$.
        \item \label{item: thm: d geq 3}
            For an integer $d \geq 3$ and a rational number $c\in \Q$, assume that $c\neq 0, -1$ if $d$ is even, and assume that $c \neq 0$ if $d$ is odd.
            Then, the morphism $f_{d,c}$ admits no rational second preimages of $0$.
    \end{parts}
\end{thm}
As a corollary, we get the following result.
\begin{cor}\label{cor: number of rational preimages}
    For even $d\geq 3$, we have the equality
    \[
        \#\bigcup_{N\geq 1}f_{d,c}^{-N}(0)(\Q)=
        \begin{cases}
            3 & \text{for } c = -1,\\
            1 & \text{for } c = 0,\\
            2 & \text{for } c = -(d\text{-th power of rationals other than } 0, \pm 1), \text{ and}\\
            0 & \text{for other rationals.}
        \end{cases}
    \]
    For odd $d \geq 3$, we have the equality
    \[
        \#\bigcup_{N\geq 1}f_{d,c}^{-N}(0)(\Q)=
        \begin{cases}
            1 & \text{for } c = -(d\text{-th power of rationals}), \text{ and}\\
            0 & \text{for other rationals}.
        \end{cases}
    \]
    Hence, we have
    \[
        \kappa(d,0,\Q) =
        \begin{cases}
            6 &\text{ for } d = 2,\\
            3 &\text{ for even } d \geq 3, \text{ and}\\
            1 &\text{ for odd } d \geq 3.
        \end{cases}
    \]
\end{cor}

We will prove \cref{thm: main} by the following strategy.
For $d=2$, we exploit the geometry of the degree $2$ cover
\[
    \pi\colon \Xpre(2,4,0)\longrightarrow \Xpre(2,3,0)\cong E
\]
with
\[
    E: v^2=u^3-u+1,
\]
analyze ramification above $E[2]$,
and pass to suitable base changes to obtain finite \'etale double covers of $\Xpre(2,4,0)$.
This reduces the problem to lower-genus curves $C_D$ classified by a class $[D]\in R^\times/(R^\times)^2$ for $R= \Z[1/2,1/23]$.
We determine $C_D(\mathbb Q)$ via the elliptic Chabauty method due to \cite{Bru03}.
\begin{thm}\label{thm: rational points of CD}
    For $D = (-1)^{\varepsilon_0} 2^{\varepsilon_1} 23^{\varepsilon_2}$ with $\varepsilon_i \in \{0, 1\}\ (0 \leq i \leq 2)$,
    let $C_D$ be the affine plane curve defined by
    \[
        D^2y^4 = x^3 - x + 1.
    \]
    Then, $C_D(\Q)$ is non-empty only if $D = \pm 1$.
    Moreover, we have
    \[
        C_1(\Q) = C_{-1}(\Q) = \{(0, \pm 1), (\pm1, \pm 1)\}.
    \]
\end{thm}
Then, we obtain the assertion in the quadratic case.

For the case $d\geq 3$, the conclusion follows even from a Diophantine reduction to the results of \cite{DM97} on the consequences of modularity.

\begin{rem}
    There is an arboreal perspective for \cref{thm: main} and the previous works \cite{FHIJMTZ09, FHS11}, and \cite{HHK11}.
    Fix $a$. The iterated preimages of $a$ by $f_{d,c}$ form a rooted $d$-ary tree $T_d$, and the absolute Galois group acts via the arboreal representation
    \[
        \rho_{f,a}: \Gal(\overline L/L)\longrightarrow\mathrm{Aut}(T_d).
    \]
    The uniformity results on the rational iterated preimages are equivalent to the fact that rational vertices cannot persist beyond a uniform level.
    This picture is consistent with the expectation that, for generic $c$, the image of $\rho_{f, a}$ is large and rational constraints become increasingly scarce deeper in the tree.
    See \cite{Odo85} and \cite{Jon13} for these expectations.
\end{rem}

\begin{rem}
    We may be able to generalize the method used in the proof of \cref{thm: rational points of CD} to the curves defined by
    \begin{equation} \label{eq:y4=fx}
        D^2 y^4 = f(x)
    \end{equation}
    for $f \in \Q[x]$ having no rational roots and $\deg f = 3$.
    Moreover, we can use the method to determine the perfect squares in elliptic divisibility sequences (see \cite{Ray12, NMSS24} for the study of perfect powers in elliptic divisibility sequences in general).
    Although generalized $n$-descent for curves of genus $3$ has already been developed in \cite{BPS16}, it needs heavy computational costs, and often needs to assume the BSD conjecture or the generalized Riemann Hypothesis.
    This is the reason that \cite{FHS11} needed the BSD conjecture.
    Our method seems to be comparatively computationally light.

    Although the explicit descent on the Jacobians of the curve
    \[
        y^p = f(x)
    \]
    for a prime number $p$ is dealt with in 
    \cite{Sch95, PS97, Sto01, BS09, Mou10, Cre12}, our curve \cref{eq:y4=fx} is out of their scope.
\end{rem}

% --------------------------------------------------------------------------

\subsection*{Organization of this paper}

% --------------------------------------------------------------------------

In \cref{subsec: geometry}, we recall some geometric facts on $\Xpre(2,N,0)$ developped in \cite{FHIJMTZ09} and \cite{FHS11}.
In \cref{subsec: arithmetic}, we list some arithmetic facts on a specific elliptic curve.
In \cref{sec: d = 2} and \cref{sec: d geq 3} involves the proof of \cref{thm: main} \cref{item: thm: d = 2} and \cref{item: thm: d geq 3}, respectively.
In \cref{subsec: integral models and ramification}, we define integral models of curves under consideration, and we investigate their singularities and \'etaleness of some morphisms.
In \cref{subsec: reduction to CD} we reduce the proof of \cref{thm: main} \cref{item: thm: d = 2} to the determination of rational points of $C_D$'s.
In \cref{subsec: rational points of CD}, we completely determine the rational points of $C_D$'s.
In \cref{subsec: proof of main theorem for d=2}, we complete the proof of \cref{thm: main}\cref{item: thm: d = 2}.

% --------------------------------------------------------------------------

\section*{Acknowledgement}

% --------------------------------------------------------------------------
The author would like to thank Professor Bjorn Poonen for providing a geometric perspective that clarified both the elementary and the more intricate calculations in this work.
The author is also grateful to Professor Shun’ichi Yokoyama for his helpful advice on the use of \textsc{Magma}, and to Professor Yasuhiro Ishitsuka and Professor Hiroyasu Miyazaki for their valuable comments.

% --------------------------------------------------------------------------

\section{Preliminaries} \label{sec: preliminaries}

% --------------------------------------------------------------------------

% --------------------------------------------------------------------------

\subsection{Geometry} \label{subsec: geometry}

% --------------------------------------------------------------------------

In this subsection, we recall some geometric facts on the preimage curves for $d = 2$ given in \cite{FHIJMTZ09} and \cite{FHS11}.
Let $k$ be a field of characteristic different from $2$.
For $a \in k$, define an algebraic set
\[
    \Ypre(2, N, a) \coloneqq V(f_{2,c}^{\circ N}(x) - a) \subset \bbA_k^2 =\Spec k[c,x],
\]
and a morphism
$\iota \colon \Ypre(2, N, a) \longrightarrow \bbA^N$ by
\[
    \iota(c,x) = (f_{2,c}^{\circ(N-1)}(x), f_{2,c}^{\circ(N-2)}(x), \ldots , f_{2,c}(x), x).
\]
Note that the order of the polynomials defining $\iota$ is the reverse of that in \cite{FHS11}.
If $\Ypre(2, N, a)$ is geometrically irreducible, let $\Xpre(2, N, a)$ be the unique complete nonsingular curve birational to $\Ypre(2, N, a)$.
We have the following lemma.

\begin{lem}[{\cite[Lemma $4.1$, Theorem $4.2$, and Corollary $4.3$]{FHS11}}]
    \noindent
    \begin{parts}
        \item \label{item: lem: defining ideal}
            The morphism $\iota \colon \Ypre(2, N, a) \longrightarrow \bbA^N$ is a closed immersion.
            If $\bbA^N$ has coordinates $z_1, z_2, \ldots, z_N$, then the ideal defining the image of $\iota$ is
            \[
                I = (z_1^2 + z_i - z_{i+1}^2 - a : 1 \leq i \leq N-1).
            \]
            
        \item \label{item: lem: defining homogeneous ideal}
            Let $W,Z_1,Z_2,\ldots,Z_N$ be the homogeneous coordinate of $\bbP^N$, and identify $\bbA^N$ with the subset of $\bbP^N$ where $W\neq 0$. Then, the projective closure of the image of $\iota$ is a complete intersection of the closed set defined by the homogeneous ideal
            \[
                J = (Z_1^2 + Z_iW - Z_{i+1}^2 - aW^2 : 1 \leq i \leq N-1).
            \]
            
        \item \label{item: lem: boundaries}
            The points of $V(J)$ on the hyperplane $W = 0$ have homogeneous coordinates
            \[
                [0, \epsilon_1, \epsilon_2, \ldots, \epsilon_N], \quad \epsilon_i = \pm 1.
            \]
            Moreover, they are all nonsingular points of $V(J)$.
            
        \item \label{item: lem: normal projective model}
            If $\Ypre(2, N, a)$ is nonsingular, then we have $\Xpre(2, N, a) \cong V(J)$ and the complement of the affine part $\Xpre(2, N, a) \setminus \Ypre(2, N, a)$ consists $2^{N-1}$ points.
            
        \item \label{item: lem: geometric irreducibility}
            If $\Ypre(2, N, a)$ is nonsingular, then the curve $\Xpre(2, N, a)$ can be defined over $k$, and it is both nonsingular and geometrically irreducible.
    \end{parts}
\end{lem}
\begin{proof}
    See \cite[Section $4$]{FHS11} for \cref{item: lem: defining ideal}--\cref{item: lem: normal projective model} and \cite[Proposition $4.8$]{FHIJMTZ09} for  \cref{item: lem: geometric irreducibility}.
    Note that the coordinates in \cite{FHS11} are chosen to satisfy $f_{2,c}^{\circ i}(z_0)=z_i$, and the coordinates in this paper are arranged to satisfy $f_{2,c}^{\circ i}(z_i) = a$.
\end{proof}

\begin{lem}\label{lem: geometry 2}
    \begin{parts}
        \item
            The complete curve $\Xpre(2,3,0)$ is isomorphic to the elliptic curve
            \[
                E: v^2 = u^3 - u + 1
            \]
            over $\Q$. The birational transform $E \dashrightarrow \Ypre(2, 3, 0)$ is given by
            \[
                x = \frac{v}{u^2 - 1}, \quad c = -\frac{1}{(u^2 - 1)^2}.
            \]
        \item\label{item: lem: isomorphisms}
            If $[U, V, S]$ is the homogeneous coordinate of $E$, the isomorphism $E\overset{\sim}{\longrightarrow}\Xpre(2,3,0)$ is defined by
            \[
                [U, V, S] \mapsto
                \begin{cases}
                    [U^2 - S^2, S^2, US, VS] & \text{if } S\neq 0,\\
                    [V^2 - S^2, US, U^2, UV] & \text{if } V^2 - S^2 \neq 0,
                \end{cases}
            \]
            and the inverse is given by
            \[
                [W, Z_1, Z_2, Z_3] \mapsto
                \begin{cases}
                    [Z_2, Z_3, Z_1] & \text{if }Z_i \neq 0 \text{ for some }i,\\
                    [(W + Z_1)Z_3, Z_1Z_2 + Z_1 W + W^2, Z_2Z_3] & \text{if } Z_1Z_2 + Z_1 W + W^2 \neq 0.
                \end{cases}
            \]
            Note that these expressions coincide with each other at the intersection of the loci.
        \end{parts}
\end{lem}
\begin{proof}
    See \cite[Lemma 5.2]{FHS11} for the statement on birational transformation.
    We can check the coincidence of the two expressions as follows
    \begin{align*}
        [U^2 - S^2, S^2, US, VS]
        &= [U^3 - US^2, US^2, U^2S, UVS]\\
        &= [V^2S - S^3, US^2, U^2S, UVS]\\
        &= [V^2 - S^2, US, U^2, UV], \text{ and}\\
        [Z_2, Z_3, Z_1]
        &= [Z_2^2Z_3, Z_2Z_3^2, Z_1Z_2Z_3]\\
        &= [(W + Z_1)Z_1Z_3, Z_2(Z_1^2 + WZ_2), Z_1Z_2Z_3]\\
        &= [(W + Z_1)Z_1Z_3, Z_1^2Z_2 + W(Z_1^2+Z_1W), Z_1Z_2Z_3]\\
        &= [(W + Z_1)Z_3, Z_1Z_2 + Z_1W + W^2, Z_2Z_3].
    \end{align*}
\end{proof}

% --------------------------------------------------------------------------

\subsection{Arithmetic} \label{subsec: arithmetic}

% --------------------------------------------------------------------------
This subsection lists some arithmetic facts on curves, which we will use in the proof of \cref{thm: main} \cref{item: thm: d = 2}.
Almost all facts in this subsection are checked using \textsc{Magma} \cite{magma}.

\begin{lem}\label{lem: number field}
    Let $K = \Q(\theta)$ be the number field generated by an element $\theta$ with the minimal polynomial 
    \[
        f(t) = t^3 - t + 1.
    \]
    Let $\calO_K$ be the ring of integers of $K$.
    Then, the following statements hold.
    \begin{parts}
        \item The discriminant of $K$ is $-23$.
        \item \label{item: lem: ramification of number field}
            The only prime ideal $23$ ramifies in $K$, and we can write $(23) = \pe_1 \pe_2^2$ with distinct non-zero prime ideals $\pe_1, \pe_2$ of $K$. Moreover, $\pe_1$ and $\pe_2$ are generated by $\beta_1 \coloneqq 3\theta^2 - 4$ and $\beta_2 \coloneqq 3\theta^2 -1$, respectively.
        \item \label{item: lem: class group}
            The class number of $K$ is $1$.
        \item \label{item: lem: unit group}
            The unit group of $\calO_K$ is generated by $\pm 1$ and $\theta$, and is isomorphic to $\Z/2\Z \times \Z$.
    \end{parts}
\end{lem}

\begin{lem}\label{lem: elliptic curve}
    Let $E$ be the elliptic curve defined by 
    \[
        E: v^2 = u^3 - u + 1
    \]
    with the point $O$ at infinity.
    Let $K$, $\theta$, $\pe_1$, and $\pe_2$ be as in \cref{lem: number field}.
    %Let $\sigma_i$ $(1 \leq i \leq 3)$ be the field embeddings of $K$ into $\bbC$.
    Then, the following statements hold.
    \begin{parts}
        \item\label{item: lem: MWG of E} The Mordell--Weil group $E(\Q)$ is isomorphic to $\Z$, and is generated by $Q_0 \coloneqq (1, -1)$.
        \item \label{item: lem: known points}
            There are $10$ known rational points in $\Xpre(2, 4, 0)$ as follows.
            \[
                \begin{array}{ll|ll}
                        &\Xpre(2, 4, 0)(\Q) & \mu \circ \pi(P_i) &= mQ_0 \\ \hline
                    P_1 & [0,1,1,1,1]      &\ [1, 1, 1]   &= - Q_0\\
                    P_2 & [0,-1,1,1,1]     &\ [-1, -1, 1] &= 2Q_0\\
                    P_3 & [0,1,-1,1,1]     &\ [-1, 1, 1]  &= -2Q_0\\
                    P_4 & [0,-1,-1,1,1]    &\ [1, -1, 1]  &= Q_0\\
                    P_5 & [0,1,1,-1,1]     &\ [1, -1, 1]  &= Q_0\\
                    P_6 & [0,-1,1,-1,1]    &\ [-1, 1, 1]  &= -2Q_0\\
                    P_7 & [0,1,-1,-1,1]    &\ [-1, -1, 1] &= 2Q_0\\
                    P_8 & [0,-1,-1,-1,1]   &\ [1, 1, 1]   &= -Q_0\\
                    P_9 & [1,0,0,0,0]      &\ [0, 1, 0]   &= O\\
                    P_{10} & [1,-1,0,-1,0] &\ [0, 1, 1]   &= 3Q_0,\\
                \end{array}
            \]
            where $\pi\colon \Xpre(2, 4, 0) \longrightarrow \Xpre(2, 3, 0)$ is the projection $[W, Z_1, Z_2, Z_3, Z_4] \mapsto [W, Z_1, Z_2, Z_3]$, and $\mu \colon \Xpre(2, 3, 0) \longrightarrow E$ is the isomorphism defined in \cref{lem: geometry 2} \cref{item: lem: isomorphisms}.

        \item \label{item: lem: ramification locus}
            The morphism $\pi$ ramifies at the points with $Z_4 = 0$. Precisely, these points are $P_9, P_{10}$, and $6$ points written as
            \[
                [1, f_{2,c_0}^{3}(0) f_{2,c_0}^2(0) , f_{2,c_0}(0), 0],
            \]
            where $c_0$'s are the roots of
            \[
                F(c) := f_{2,c}^{\circ 4}(0)/f_{2,c}^{\circ 2}(0) = c^6 + 3c^5 + 3c^4 + 3c^3 + 2c^2 + 1 = 0.
            \]
            
        \item \label{item: lem: discriminant}
            The discriminant of $F$ is $58673 = 23 \cdot 2551$.
            Moreover, we have
            \[
                F(c) =
                \begin{cases}
                    (c+4)^2 (c+18) (c^3 + 4c + 2) & \text{over } \mathbb{F}_{23}\\
                    (c+477)^2(c^4 + 1600c^3 + 1162c^2 + 297c + 1869) & \text{over } \mathbb{F}_{2551}.
                \end{cases}
            \]
        \item \label{item: lem: x-T map}
            Define the map $(x - T) \colon E(\Q) \longrightarrow K^\times / (K^\times)^2$ by
            \[
                P \mapsto
                \begin{cases}
                    (u - \theta)(K^\times)^2 & \text{for } P = (u,v) \neq O,\\
                    (K^\times)^2             & \text{for } P = O.
                \end{cases}
            \]
            Then, the map $(x - T)$ is a homomorphism with the kernel $2E(\Q)$, and its image equals
            \[
                \{ (K^\times)^2, -\theta(K^\times)^2 \}.
            \]
    \end{parts}
\end{lem}

% --------------------------------------------------------------------------

\section{Proof of main results for $d = 2$} \label{sec: d = 2}

% --------------------------------------------------------------------------
In this section, we prove \cref{thm: main} \cref{item: thm: d = 2}.
To avoid the calculation of the Mordell-Weil rank of the Jacobian of $\Xpre(2,4,0)$, we efficiently use the geometric information of the surjection $\pi \colon \Xpre(2,4,0)\longrightarrow \Xpre(2,3,0)$ and a tricky argument based on the $2$-descent of the elliptic curve
\[
    E: v^2 = u^3 - u + 1,
\]
which is isomorphic to $\Xpre(2,3,0)$.

% --------------------------------------------------------------------------

\subsection{Integral models and their geometry} \label{subsec: integral models and ramification}

% --------------------------------------------------------------------------

In this section, we study the models of $\Xpre(2,4,0)$ and $E$ over $R = \Z[\frac{1}{2},\frac{1}{23}]$.
Consider the projective $R$-schemes $\calE$, $\calF$ $\calX_3$, and $\calX_4$ defined by the following  homogeneous ideals
\begin{align*}
    I_\calE &= (V^2 S - U^3 + US^2 - S^3) \subset R[U, V, S],\\
    I_\calF &= (T^4 - U^3 S + US^3 - S^4) \subset R[U, T, S],\\
    I_{\calX_3} &= (Z_1^2 + Z_iW - Z_{i+1}^2 : 1 \leq i \leq 2) \subset R[W, Z_1, Z_2, Z_3],\text{ and}\\
    I_{\calX_4} &= (Z_1^2 + Z_iW - Z_{i+1}^2 : 1 \leq i \leq 3) \subset R[W, Z_1, Z_2, Z_3, Z_4],
\end{align*}
respectively.
Note that $\calE$ and $\calF$ are smooth (see \cref{lem: ramification and singular locus} \cref{item: lem: smoothness of E F}), and the generic fiber of $\calE$ equals $E$ in \cref{lem: elliptic curve}.
The self-morphisms $[2]$ and $T_{3Q_0}$ on $E$ can be extended to $\calE$.
We use the same notation for the extended morphisms.

\if0
%%%%%%%%%%%%%%%%%%%%%%%%%%%%%
% We don't use this.
%%%%%%%%%%%%%%%%%%%%%%%%%%%%%
Here, we give the explicit expressions of these morphisms since we will use them in the proof of \cref{lem: ramification and singular locus}.

\begin{align*}
    [2] \colon &[U,V,S] \mapsto
    \begin{cases}
        \begin{array}{ll}
            [2(U^4 + 2U^2 S^2 - 8US^3 + S^4)(U^3 - US^2 + S^3) ,\\
            \ (U^6 - 5U^4S^2 + 20U^3S^3 - 5U^2S^4 + 4US^5 - 7S^6)V,\\
            \ 8(U^3 - US^2 + S^3)^2 S]& \text{if }V \neq 0,
        \\
        [6U^2VS + 2(-9S^2+V^2)UV + 2VS^3,\\
        \ -9U^2S^2 + US(27S^2-3V^2) + (-26S^4 + 18V^2S^2 + V^4),\\
        \ 8V^3S)] & \text{around } [0,1,0], \text{and}
        \end{array}
    \end{cases}\\
    T_{3Q_0} \colon &[U,V,S] \mapsto
    \begin{cases}
        \begin{array}{ll}
        [US(-U - 2V + 2S),\\
        \ (U - 4S)VS + (U^3 - 3US^2 + 4S^3),\ U^3] &\text{if }[U,V,S] \neq [0,1,0], [0,1,1]\\
        [(-U^2 + S^2 - 2U(V + S))(V + S),\\
        \ -U^2S + (V^2 - S^2)U+(-2S^3 - 5VS^2 - V^2S + V^3),\\
        \ (V + S)^3] & \text{if } V+S \neq 0
        \end{array}
    \end{cases}
\end{align*}
%%%%%%%%%%%%%%%%%%%
%%%%%%%%%%%%%%%%%%%
\fi

Let $\varphi \colon \calE \longrightarrow \calE$ be either $[2]$ or $T_{3Q_0}\circ [2]$.
Let $\pi$ be the composition of the morphisms $\pr\colon \calX_4 \longrightarrow \calX_3$ and $\mu \colon \calX_3 \overset{\sim}{\longrightarrow} \calE$ defined by
\begin{align*}
    \pr \colon &[W,Z_1,Z_2,Z_3,Z_4] \mapsto [W,Z_1,Z_2,Z_3], \text{ and}\\
    \mu \colon &[W,Z_1,Z_2,Z_3] \mapsto
    \begin{cases}
        [Z_2, Z_3, Z_1] & \text{if }Z_i \neq 0 \text{ for some }i,\\
        [(W + Z_1)Z_3, Z_1Z_2 + Z_1 W + W^2, Z_2Z_3] & \text{if } Z_1Z_2 + Z_1 W + W^2 \neq 0,
    \end{cases}
\end{align*}
which is dealt with in \cref{lem: geometry 2}. Let $\pi$ be the composition $\mu\circ \pr$. Note that $\mu$ is an isomorphism of schemes over $R$.
Let $\psi \colon \calF \longrightarrow \calE$ be the morphism defined by
\begin{align*}
    \psi \colon &[U,T,S] \mapsto
    \begin{cases}
        [US, T^2, S^2] & \text{if } S \neq 0,\\
        [UT^2, U^3 - US^2 + S^3, ST^2] & \text{if } U^3 - US^2 + S^3 \neq 0.
    \end{cases}
\end{align*}

Let $\calX_4'$ and $\calX_4''$ be the projective schemes over $R$ given by the following Cartesian diagrams.
\[
  \begin{CD}
     \calX_4 @<{\varphi'}<< \calX_4' @<{\psi'}<< \calX_4''\\
        @V{\pi}VV        @V{\pi'}VV        @VV{\pi''}V\\
     \calE   @<{\varphi}<<     \calE   @<{\psi}<< \calF
  \end{CD}
\]

Before proving the flatness of each scheme, we provide sufficient conditions for the flatness of algebras over a principal ideal domain.

\begin{lem}\label{lem: criterion of flatness}
    Let $R_0$ be a ring and $n$ and $m$ be positive integers with $m\leq n$.
    For polynomials $F_i(T_i, T_{i+1},\ldots, T_n)$ $(1\leq i \leq m)$ in $R' = R_0[T_1, T_2, \ldots, T_n]$,
    suppose that $F_i$ is monic in $T_i$ up to multiplying a unit in $R_0$.
    Let $I$ be the ideal $(F_1,F_2, \ldots, F_m)$.
    Then, the following statements hold.
    \begin{parts}
        \item \label{item: lem: element is in ideal}
            For $g\in R'$, set $r_0 = g$, and $r_{i+1}$ for $0\leq i \leq m - 1$ be the remainder of $r_{i}$ by $F_{i+1}$ as polynomials in the variable $T_{i+1}$.
            Then, $g$ is in $I$ if and only if $r_m = 0$ holds.
        \item \label{item: lem: flatness of the algebra}
            Suppose that $R_0$ is a principal ideal domain.
            Then, the $R_0$-algebra $R'/I$ is flat.
    \end{parts}
\end{lem}
\begin{proof}
    \cref{item: lem: element is in ideal}
    If $r_m = 0$, the polynomial $g$ is obviously in $I$.
    
    Assume that $g$ is in $I$.
    Then, there are polynomials $g_{i,0}\in R'$ $(1\leq i \leq m)$ such that
    \[
        r_0 = g = \sum_{i=1}^m g_{i,0} F_i.
    \]
    Divide each $g_{i,0}$ for $2\leq i \leq m$ by $F_1$ as polynomials in the variable $T_1$, and write the quotient $q_{i,1}$ and the remainder $g_{i,1}$.
    Set $g_{1,1} = \sum_{i=2}^m q_{i,1} + g_{1,0}$.
    Then, we have the equalities
    \begin{equation*}
        r_0 = \sum_{i=1}^m g_{i,1} F_i, \text{ and} \quad
        r_1 = \sum_{i=2}^m g_{i,1} F_i.
    \end{equation*}
    By inductively running the same process replacing the role of $r_0$ and $F_1$ with $r_j$ and $F_{j+1}$ for $j= 1,2,\ldots m-1$, we get the polynomials $g_{i,j} \in R'$ $(j+1 \leq i \leq m)$ such that
    \begin{equation*}
        r_j = \sum_{i = j+1}^m g_{i,j+1} F_i, \text{ and} \quad 
        r_{j+1} = \sum_{i = j+2}^m g_{i,j+1} F_i.
    \end{equation*}
    This observation shows that $r_m = 0$.

    \cref{item: lem: flatness of the algebra}
    Since the flatness over a principal ideal domain $R_0$ is equivalent to $R_0$-torsion freeness, it is enough to show that for $g\in R'$, there exists an element $a\in R_0 \setminus \{0\}$ such that $ag$ is in $I$ if and only if $g$ is itself in $I$.
    Suppose that $ag$ is in $I$ for some $a\in R_0\setminus\{0\}$.
    Set $r_i$ and $r_i'$ for $1\leq i \leq m$ as in the statement of \cref{item: lem: element is in ideal} for $g$ and $ag$, respectively.
    Then, we have $r_i' = ar_i$.
    Thus, $r_m' = 0$ if and only if $r_m = 0$.
    By \cref{item: lem: element is in ideal}, we obtain the assertion.
\end{proof}

\begin{lem}\label{lem: flatness of each schemes}
    \begin{parts}
        \item \label{item: lem: flatness of X4}
            The $R$-scheme $\calX_4$ is flat over $R$.
        \item \label{item: lem: flatness of E}
            The $R$-scheme $\calE$ is flat over $R$.
        \item \label{item: lem: flatness of F}
            The $R$-scheme $\calF$ is flat over $R$.
    \end{parts}
\end{lem}
\begin{proof}
    \cref{item: lem: flatness of X4}
    Two affine open subschemes
    \begin{align*}
        \calX_4 \cap D_+(W) &\cong \Spec R[Z_1,Z_2,Z_3,Z_4]/(Z_{i+1}^2 - Z_1^2 - Z_i: 1\leq i\leq 3), \text{ and}\\
        \calX_4 \cap D_+(Z_1) &\cong \Spec R[Z_2, Z_3, Z_4,W]/(Z_4^2 - Z_3W-1, Z_3^2 - Z_2W - 1, Z_2^2 - W - 1).
    \end{align*}
    covers the $R$-scheme $\calX_4$.
    Three affine open subschemes
    \begin{align*}
        \calE \cap D_+(U) &\cong \Spec R[V, S]/(V^2S - S^3 + S^2 - 1),\\
        \calE \cap D_+(V) &\cong \Spec R[U, S]/(S - U^3 + US^2 - S^3),\text{ and}\\
        \calE \cap D_+(S) &\cong \Spec R[U, V]/(V^2 - U^3 + U - 1).
    \end{align*}
    cover the scheme $\calE$.
    Three affine open subschemes
    \begin{align*}
        \calF \cap D_+(U) &\cong \Spec R[T, S]/(T^4 - S + S^3 - S^4),\\
        \calF \cap D_+(T) &\cong \Spec R[U, S]/(1 - U^3 S + US^3 - S^4), \text{ and}\\
        \calF \cap D_+(S) &\cong \Spec R[U, T]/(T^4 - U^3 + U - 1).
    \end{align*}
    cover the scheme $\calF$. All of them are flat over $R$ by \cref{lem: criterion of flatness} \cref{item: lem: flatness of the algebra}.
\end{proof}

For an $R$-scheme $\calX$ and a prime ideal $(p)$ of $R$, let $\kappa(p)$ denote the residue field of $R$ at $(p)$, $\calX_p$ denote the fiber $\calX \times_{\Spec R} \Spec \kappa(p)$, and  $\calX_{\overline{p}}$ denote the scheme $\calX_p \times_{\kappa(p)} \Spec \overline{\kappa(p)}$, where $\overline{\kappa(p)}$ is the algebraic closure of $\kappa(p)$.
For an $R$-morphism $\phi \colon \calX \longrightarrow \calY$ of $R$-schemes $\calX$ and $\calY$, let $\phi_p\colon \calX_p \longrightarrow \calY_p$ and $\phi_{\overline{p}} \colon \calX_{\overline{p}} \longrightarrow \calY_{\overline{p}}$ be the induced morphisms.

\begin{lem}\label{lem: criterion of etaleness}
    For a flat finite $R$-morphism $\phi \colon \calX \longrightarrow \calY$ of $R$-schemes,
    $\phi$ is \'etale if and only if for all maximal ideal $(p) \in \Spec R$, the morphism $\phi_{\overline{p}}$ is \'etale.
\end{lem}
\begin{proof}
    Assume that $\phi_{\overline{p}}$ is \'etale for all maximal ideal $(p)\in \Spec R \setminus \{(0) \}$.
    Since $\phi_{\overline{p}}$ is \'etale, it is smooth, and so is $\phi_p$ by the definition of smoothness.
    By the finiteness of $\phi_{\overline{p}}$, it is \'etale.
    Since the \'etaleness is an open condition, $\phi$ is itself \'etale.
    
    The opposite implication is true because the base change of an \'etale morphism is \'etale.
\end{proof}

On the ramification and the singular locus, we have the following lemma.

\begin{lem}\label{lem: ramification and singular locus}
    Let notation be as above. The following statements hold.
    \begin{parts}
        \item \label{item: lem: smoothness of E F}
            The schemes $\calE$ and $\calF$ are smooth.
        
        \item \label{item: lem: singular locus of calX4}
            For all $(p) \in \Spec R \setminus \{ (2551) \}$, the scheme $(\calX_4)_{\overline{p}}$ is regular.
            The scheme $(\calX_4)_{\overline{2551}}$ is singular only at $[1,-308,13,-477,0] \in (\calX_4)_{\overline{2551}}(\overline{\bbF_{2551}})$.
            In particular, the singular point is outside $\pi_{\overline{p}}^{-1}(O) \cup \pi_{\overline{p}}^{-1}(3Q_0)$, where note that we abuse the notation for the $\bbF_p$-rational points $O = [0,1,0]$ and $3Q_0 = [0,1,1]$ over all primes $p \in \mathcal{P}$.

        \item \label{item: lem: X2551 is nodal}
            The singular point $[1,-308,13,-477,0] \in (\calX_4)_{\overline{2551}}$ is an ordinary double point.
        
        \item \label{item: lem: ramification locus and degree of pi}
            For all $(p)\in \Spec R$ and all closed point $x \in \pi_{\overline{p}}^{-1}(O) \cup \pi_{\overline{p}}^{-1}(3Q_0)$, the ramification degree of $\pi_{\overline{p}}$ at $x$ equals $2$.

        \item \label{item: lem: genus of X4 over p}
            The genus of $(\calX_4)_{p}$ is $5$ for all $(p)\in \Spec R \setminus\{(2551)\}$, and is $4$ for $(p)=(2551)$.
            
        \item\label{item: lem: integrality of X4' X4''}
            For all $(p)\in \Spec R$, the schemes $(\calX'_4)_p$ and $(\calX''_4)_p$ are integral.
            
        \item \label{item: lem: singular locus of calX4'}
            For all $(p)\in \Spec R$, the scheme $(\calX_4')_{\overline{p}}$ is regular at all closed point $x' \in \pi_{\overline{p}}'^{-1}(\calE_{\overline{p}}[2])$.
        
        \item \label{item: lem: ramification locus and degree of pi'}
            For all $(p)\in \Spec R$ and all closed point $x' \in \pi_{\overline{p}}'^{-1}(\calE_{\overline{p}}[2])$, the ramification degree of $\pi'_{\overline{p}}$ at $x'$ is $2$.

        \item \label{item: lem: ramification locus and degree of phi'}
            The morphism $\psi_{\overline{p}}$ is \'etale outside $\calE_{\overline{p}}[2]$ for all $(p)\in \Spec R$.
            For all closed point $\frakf \in \psi_{\overline{p}}^{-1}(\calE_{\overline{p}}[2])$, the ramification degree of $\psi_{\overline{p}}$ at $f$ equals $2$.        
    \end{parts}
\end{lem}

\begin{proof}
\cref{item: lem: smoothness of E F}
    We can check the regularity of $\calE$ and $\calF$ at the closed point by the Jacobian criterion.
    Since $\calE$ and $\calF$ are flat over $R$, they are regular at all points.

\cref{item: lem: singular locus of calX4}
    Since $\calE$ is smooth, the curve $(\calX_4)_{\overline{p}}$ is smooth at all points other than the points at which $\pi_{\overline{p}}$ ramifies.
    Ramification points of $\pi_{\overline{p}}$ is given as in \cref{lem: elliptic curve} \cref{item: lem: ramification locus} over all $p\in \calP$.
    In particular, these points are in the open subset $(W\neq 0)$.
    This open subset is isomorphic to the affine scheme $\Ypre(2,4,0)_{\overline{p}} = V(f_{2,c}^{\circ 4}(x)) \subset (\bbA^2_{x,c})_{\overline{p}}$.
    By the Jacobian criterion, a point $(0,c_0)$ is a singular point of $\Ypre(2,4,0)_{\overline{p}}$ if and only if $c_0$ is a multiple root of $f_{2,c}^{\circ 4}(0)$.
    Such $c_0$ exists only over $2551$ by \cref{lem: elliptic curve} \cref{item: lem: discriminant}.
    Moreover, such $c_0$ for $p=2551$ is $-477$.
    The point $(x,c) = (0,-477)$ corresponds to $[1,-308,13, -477,0] \in (\calX_4)_{\overline{2551}}(\overline{\bbF_{2551}})$.

\cref{item: lem: X2551 is nodal}
    The direct calculation shows that the Hessian of $f_{2,c}^{\circ{4}}(x)$ at $(x,c) = (0, -477)$, is non-degenerate.
    Hence, it is an ordinary double point of $\Ypre(2,4,0)_{\overline{2551}}$ (see \cite[{Tag 0C49}]{stacks-project}).

\cref{item: lem: ramification locus and degree of pi}
    Since the sets $\pi_{\overline{p}}^{-1}(O)$ and $\pi_{\overline{p}}^{-1}(3Q_0)$ consists of only one smooth point $P_9 = [1,0,0,0,0]$ and $P_{10} = [1,-1,0,-1,0]$, respectively, and the degree of $\pi_{\overline{p}}$ is $2$, the statement holds.

\cref{item: lem: genus of X4 over p}
    For $p \in \calP \setminus\{2551\}$, the curve $(\calX_4)_{\overline{p}}$ is smooth, and $\pi_{\overline{p}}$ ramifies at $8$ points as similar to \cref{lem: elliptic curve} \cref{item: lem: ramification locus}.
    Hence, the Riemann-Hurwitz formula shows that the genus of $\calX_{\overline{p}}$ is $5$.
    
    Let $\nu \colon \widetilde{\calX}\longrightarrow (\calX_4)_{\overline{2551}}$ be the normalization of $(\calX_4)_{\overline{2551}}$.
    It is enough to show that the genus of $\widetilde{\calX}$ is $4$.
    Since the singular point of $(\calX_4)_{\overline{2551}}$ is a single ordinary double point by \cref{item: lem: X2551 is nodal}, its preimage by $\nu$ consists of two points.
    Hence, the morphism $\nu_{\overline{2551}} \circ \pi$ ramifies at $6$ points, and is of degree $2$.
    The Riemann-Hurwitz shows that the genus of $\widetilde{\calX}$ is $4$.

\cref{item: lem: integrality of X4' X4''}
    Since the extensions of function fields induced by $\varphi_p$ and $\varphi_p \circ \psi_p$ are separable extensions, the function fields of $(\calX_4')_p$ and $(\calX_4'')_p$ are reduced.
    
    Assume that $(\calX_4'')_p$ (resp. $(\calX_4')_p$) is not irreducible, and let $\calC$ be an irreducible component.
    Since the morphism $\pi''$ (resp. $\pi'$) is of degree $2$, and its restriction to each irreducible component is dominant, $\calC$ is birational to $\calF_{p}$ (resp. $(\calE)_p$).
    Hence, the genus of $\calC$ is $3$ (resp. $1$).
    However, since the restriction of the morphism $\varphi'_p\circ \psi'_p$ (resp. $\varphi'_p$) to $\calC$ is dominant, the genus of $\calC$ is more than or equal to that of $(\calX_4)_p$, which is $4$ or $5$ by \cref{item: lem: genus of X4 over p}.
    This is a contradiction.

\cref{item: lem: singular locus of calX4'}
    Since $\varphi$ is \'etale, the morphism $\varphi_{\overline{p}}$ is \'etale, in particular smooth.
    Since $\varphi_{\overline{p}}(x')$ is a regular point of $(\calX_4)_{\overline{p}}$ by \cref{item: lem: singular locus of calX4}, the assertion holds (see \cite[Chapter 4, Theorem 3.36]{Liu02}).

\cref{item: lem: ramification locus and degree of pi'}
    Let $x'$ be an element of the fiber $\pi_{\overline{p}}'^{-1}(e)$ for some $e\in \calE_{\overline{p}}[2](\overline{\kappa(p)})$. Note that $x'$ is a regular point by \cref{item: lem: singular locus of calX4'}.
    The set $\pi_{\overline{p}}^{-1}(\varphi_{\overline{p}}(e))$ consists of a single regular point, either $P_9$ or $P_{10}$.
    Since the scheme $\calX_{\overline{p}}$ is defined as the fiber product, the point $x'$ is the unique point of $\pi_{\overline{p}}'^{-1}(e)$.
    Therefore, the ramification degree of $\pi_{\overline{p}}'$ at $x'$ equals the degree of $\pi_{\overline{p}}'$, which is $2$.

\cref{item: lem: ramification locus and degree of phi'}
    For $e\in (\calE_{\overline{p}}\setminus \calE_{\overline{p}}[2])(\overline{\bbF_{p}})$, the fiber $\psi_{\overline{p}}'^{-1}(e)$ consists of two regular points.
    For $e\in \calE_{\overline{p}}[2]$, the fiber $\psi_{\overline{p}}'^{-1}(e)$ consists of a single regular point.
    Since the degree of $\psi_{\overline{p}}'$ is $2$, the assertion holds.
\end{proof}

\begin{prop}\label{prop: tildepsi' is etale}
    Let notation be as above.
    Let $\widetilde{\calX_4}''$ be %the relative normalization $\underline{\Spec}_{\calX_4'}((\psi'_\ast \calO_{\calX''_4})^{\mathrm{nor}})$,
    the partial normalization of $\calX_4''$ at $(\pi'\circ \psi')^{-1}(\calE[2])$, i.e., the scheme given by glueing $\calX_4''\setminus (\pi'\circ \psi')^{-1}(\calE[2])$ and the normalization of $\calX_4'' \setminus \left(\Sing(\calX_4'')\setminus(\pi'\circ \psi')^{-1}(\calE[2])\right)$.
    Let $\nu \colon \widetilde{\calX_4}'' \longrightarrow \calX_4''$ be the canonical morphism.
    Then, the morphism $\psi' \circ \nu \colon \widetilde{\calX_4}'' \longrightarrow \calX_4'$ is finite \'etale of degree $2$.
\end{prop}
\begin{proof}
    Only the \'etaleness is non-trivial.
    By \cref{lem: criterion of etaleness}, it suffices to check \'etaleness fiberwise.
    Let $\eta, \xi$, and $\eta'$ be the generic points of $\calE[2], \pi'^{-1}(\calE[2])$, and $\psi^{-1}(\calE[2])$, respectively.
    Then, the rings
    \begin{align*}
        A \coloneqq \calO_{\calE,\eta}, \quad
        B \coloneqq \calO_{\calX_4',\xi}, \text{ and} \quad 
        A_1 \coloneqq \calO_{\calF,\eta'}
    \end{align*}
    are DVR. Moreover, since the ramification degrees of the extensions $A_1/ A$ and $B/A$ are both equal to $2$.
    Hence, letting $B_1$ be the integral closure of $B\otimes_A A_1$ in the fraction field at a generic point of $(\pi' \circ \psi' \circ \nu)^{-1}(\calE[2])$, we can see that $B_1/ B$ is unramified by the Abhyankar's Lemma \cite[Tag 0BRM]{stacks-project}.
    Hence, the morphism $\psi'\circ \nu$ is \'etale on $(\pi'\circ \psi' \circ \nu)^{-1}(\calE[2])$ by the purity of the branch locus (see \cite[Tag 0BMB]{stacks-project}).
    
    Let $U$ be the open subscheme $\calX_4'' \setminus (\pi' \circ \psi')^{-1}(\calE[2])$.
    Then, restriction of $\nu$ to $\nu^{-1}(U)$ is isomorphism.
    Since $\psi'$ is \'etale on $U$, the morphism $\psi'\circ \nu$ is \'etale on $\nu^{-1}(U)$.
    Consequently, $\psi' \circ \nu$ is \'etale on the whole space $\widetilde{\calX_4}''$.
    \end{proof}

% --------------------------------------------------------------------------

\subsection{Reduction to curves $C_D$ of lower genus}\label{subsec: reduction to CD}
Let $X_4 = \Xpre(2,4,0)$, $X_4'$, $X_4''$, $\widetilde{X_4}''$, $F$ be the generic fiber of $\calX_4$, $\calX_4'$, $\calX_4''$, $\widetilde{\calX_4}''$, and $F$ respectively.
Then, we have the diagram
\begin{equation*}
  \begin{CD}
     X_4 @<{\varphi'}<< X_4' @<{\psi'}<< X_4'' @<{\nu}<< \widetilde{X_4}''\\
        @V{\pi}VV        @V{\pi'}VV        @VV{\pi''}V\\
     E   @<{\varphi}<<     E   @<{\psi}<< F
  \end{CD}
\end{equation*}
commutes.
For a rational point $x \in X_4(\Q)$, the point $\pi(x)$ is in the image $\varphi(E(\Q))$ for either $\varphi = [2]$ or $T_{3Q_0}\circ [2]$ since the Mordell-Weil group $E(\Q)$ is generated by $Q_0$ by \cref{lem: elliptic curve}\cref{item: lem: MWG of E}.
Thus, it lifts to $x' \in X_4'(\Q)$.
By the projectivity of $\calX_4'$, the point $x'$ lifts to a section $x' \in \calX_4'(\Spec R)$, where $R = \Z[\frac{1}{2}, \frac{1}{23}]$.
Taking the fiber product, we get the following diagram:

\begin{center}
\begin{tikzcd}\label{eq:diagram2}
	{\Spec R} & {\Spec R} & & {\Spec R'} \\
	{\mathcal{X}_4} & {\mathcal{X}_4'} & {\mathcal{X}_4''} & {\widetilde{\mathcal{X}_4}''} \\
	{\mathcal{E}} & {\mathcal{E}} & {\mathcal{F}}
	\arrow["x"', from=1-1, to=2-1]
	\arrow[from=1-2, to=1-1]
	\arrow["{x'}"', from=1-2, to=2-2]
	\arrow["\varpi"',from=1-4, to=1-2]
	\arrow["\lrcorner"{anchor=center, pos=0.125, rotate=-90}, draw=none, from=1-4, to=2-2]
	\arrow["{\widetilde{x}''}"', from=1-4, to=2-4]
	\arrow["\pi"', from=2-1, to=3-1]
	\arrow["{\varphi'}"', from=2-2, to=2-1]
	\arrow["\lrcorner"{anchor=center, pos=0.125, rotate=-90}, draw=none, from=2-2, to=3-1]
	\arrow["{\pi'}"', from=2-2, to=3-2]
	\arrow["{\psi'}"', from=2-3, to=2-2]
	\arrow["\lrcorner"{anchor=center, pos=0.125, rotate=-90}, draw=none, from=2-3, to=3-2]
	\arrow["{\pi''}"', from=2-3, to=3-3]
	\arrow["\nu"', from=2-4, to=2-3]
	\arrow["\varphi"', from=3-2, to=3-1]
	\arrow["\psi"', from=3-3, to=3-2]
\end{tikzcd}
\end{center}
By \cref{prop: tildepsi' is etale}, the morphism $\psi'\circ \nu$ is finite \'etale of degree $2$.
Hence, the morphism $\varpi$ is finite \'etale of degree $2$.
The following proposition lets us write $R'$ down explicitly.

\begin{prop}\label{prop: classification of etale algebra}
    Let $R= \Z[1/p_i: 1 \leq i \leq n]$ with distinct prime numbers $p_1 = 2, p_2,\ldots, p_n$.
    For $D = (-1)^{\varepsilon_0}\prod_{1\leq i \leq n}p_i^{\varepsilon_i}$ with some $\varepsilon_i \in \{0,1\}$ $(0\leq i \leq n)$, let
    $R_D \coloneqq R[x]/(x^2 - D)$.
    Then, $R_D$ is a finite \'etale $R$-algebra of degree $2$.
Moreover, any finite \'etale $R$-algebra of degree $2$ is given in this way.
\end{prop}

\begin{proof}
At first, it is easy to see that the $R$-algebra $R_D$ in the statement is finite \'etale of degree $2$ over $R$ since $2$ is invertible in $R$.

Any finite \'etale covering of $\Spec R$ is automatically a Galois \'etale covering of $\Spec R$ with Galois group $\Z/2\Z$, thus, a $\mu_2$-torsor over $\Spec R$.
The group $H^1_{\text{\'et}}(\Spec R, \mu_2)$ classifies $\mu_2$-torsors over $\Spec R$ up to ismorphisms.
We have the following exact sequence of \'etale sheaves on $\Spec R$:
\[
    1 \longrightarrow \mu_2 \longrightarrow \bbG_{m} \overset{(\cdot)^2}{\longrightarrow} \bbG_m \longrightarrow 1
\]
since $2$ is invertible in $R$.
Taking the long exact sequence, we obtain the exact sequence
\[
    R^{\times} \overset{(\cdot)^2}{\longrightarrow} R^{\times} \longrightarrow H^1_{\text{\'et}}(\Spec R, \mu_2) \longrightarrow H^1_{\text{\'et}}(\Spec R, \bbG_m).
\]
Here, we note that $H^1_{\text{\'et}}(\Spec R, \bbG_m) \cong \Pic(R)$ by \cite[Proposition $6.6.1$]{Poo17}, and it is a trivial group since $R$ is a principal ideal domain.
Therefore, the group $R^{\times}/(R^{\times})^{2}$ classifies finite \'etale coverings of $\Spec R$ of degree $2$.
Each $[D] \in R^{\times}/(R^{\times})^{2}$ is represented by $(-1)^{\varepsilon_0}\prod_{1\leq i \leq n}p_i^{\varepsilon_i}$ for some $\varepsilon_i \in \{0,1\}\ (0\leq i \leq n)$.
Finally, since $R_D$'s for distinct $[D]$ are not isomorphic to each other, we obtain $2^{n+1}(=\# R^{\times}/(R^\times)^2)$ isomorphic classes of finite \'etale $R$-algebras of degree $2$.
Hence, these $R_D$'s represent all the isomorphic classes.
\end{proof}

Now, let $D$ be as in \cref{prop: classification of etale algebra} such that $R' \cong R_D$.
If $D = 1$, taking the generic fiber, we conclude that $\mathfrak{f} \coloneq \pi''\circ \widetilde{\nu}''\left(\widetilde{x}''\right)$ is an element of $F(\Q\times \Q) = F(\Q) \times F(\Q)$ such that $\psi(\mathfrak{f})$ is in $E(\Q)$, which is equivalent to that $\psi(\mathfrak{f})$ is in the diagonal of $E(\Q\times \Q) = E(\Q) \times E(\Q)$.

If $D \neq 1$, taking the generic fiber, we conclude that $\mathfrak{f} \coloneq \pi''\circ \widetilde{\nu}''\left(\widetilde{x}''\right)$ is an element of $F(\Q(\sqrt{D}))$ such that $\psi(\mathfrak{f})$ is in $E(\Q)$.
If $\mathfrak{f}$ is a point at infinity, it equals $[1:0:0]$.
Its image in $E$ under $[2]\circ \psi$ and $T_{3Q_0}\circ [2]\circ \psi$ are $O$ and $3Q_0$, respectively.
Assume that $\mathfrak{f}$ is in the affine part of $F$.
Letting $\mathfrak{f} = (u,t')$, $\psi(\mathfrak{f})$ is in $E(\Q)$ if and only if $u$ and $t'^2$ are rational numbers.
Thus, $t'$ is a multiple of a rational number and $\sqrt{D}$, or is itself a rational number.
If $t'$ is a rational number, then $(u,t')$ is a rational point of $F = C_1$.
If $t'$ is a multiple of a rational number $t$ and $\sqrt{D}$, then $(u,t)$ is a $\Q$-rational point of the curve
\[
    C_D : D^2 t^4 = u^3 - u + 1.
\]
Let $\psi_D \colon C_D \longrightarrow E; (u,t) \mapsto (u,Dt^2)$.

Consequently, for any rational point $x\in X_4(\Q)$, the point $\pi(x)$ equals $\varphi \circ \psi_D(u,t)$ for some $D = (-1)^{\varepsilon_0} 2^{\varepsilon_1} 23^{\varepsilon_2}$, $(u,t) \in C_D(\Q)$, and $\varphi = [2]$ or $T_{3Q_0}\circ [2]$.
Hence, if we obtain the complete list of the rational points of $C_D$, we can determine all the rational points of $X_4$.

% --------------------------------------------------------------------------

\subsection{Rational points of $C_D$}\label{subsec: rational points of CD}

% --------------------------------------------------------------------------
This subsection aims to prove \cref{thm: rational points of CD}.
Let $K$, $\calO_K$, $\theta$, $\pe_1 = \beta_1 \calO_K$, and $\pe_2 = \beta_2 \calO_K$ be as in \cref{lem: number field}.
This section focuses on the determination of the rational points of the affine plane curves
\[
    C_D \colon D^2 y^4 = x^3 - x + 1
\]
for $D = (-1)^{\varepsilon_0} 2^{\varepsilon_1} 23^{\varepsilon_2}$ with $\varepsilon_i \in \{0,1\}$ $(0\leq i \leq 2)$.
Here, note that we change the notation of variables from the previous subsections to save symbols.

\begin{proofwithoutqed}[Proof of \cref{thm: rational points of CD}]
For $(x,y) \in C_D(\Q)$, since we have
\[
    (D^2 y)^4 = (D^2 x)^3 - D^4 (D^2 x) + D^6,
\]
there are integers $X$, $Y$, and $Z > 0$ with
\[
    D^2y = Y/Z^3,\ D^2 x = X/Z^4, \text{ and } (X, Y) = (X, Z) = 1.
\]
The equation
\begin{equation} \label{eq:eliminatingdenominators}
    Y^4 = X^3 - D^4 XZ^8 + D^6 Z^{12}
\end{equation}
holds. Put
\[
    A \coloneqq X - D^2\theta Z^4, \quad
    B \coloneqq X^2 + D^2\theta X Z^4 + D^4 (\theta^2 - 1)Z^8,
\]
and write
\[
    A = D_A s_A^4, \quad
    B = D_B s_B^4
\]
with $D_A, D_B, s_A, s_B \in \calO_K \setminus\{0\}$ such that
$D_A$ and $D_B$ are fourth power free, and $D_A D_B$ is in $(K^\times)^4$.
Then, we have the following lemma.
\end{proofwithoutqed}

\begin{lem} \label{lem: determination of deltaA and deltaB}
    Let notation be as above.
    \begin{parts}
        \item \label{item: lem: prime ideals dividing delta}
            If a maximal ideal $\pe$ of $\calO_K$ divides $D_A$ or $D_B$,
            it divides all $D_A$, $D_B$, $D$ and $X$.
        \item \label{item: lem: 2 does not divide deltaA deltaB}
            $2\calO_K$ does not divide $D_A$ nor $D_B$.
        \item \label{item: lem: pe1 does not divide deltaA deltaB}
            $\pe_1$ does not divide $D_A$ nor $D_B$.
        \item \label{item: lem: pe2 does not divide deltaA deltaB}
            $\pe_2$ does not divide $D_A$ nor $D_B$.
        \item \label{item: lem: candidates of deltaA deltaB}
            We have $(D_A,D_B) = (1,1)$, $(-\theta, -\theta^3)$, $(\theta^2, \theta^2)$, or $(-\theta^3, -\theta)$
    \end{parts}
\end{lem}

\begin{proof}
    \cref{item: lem: prime ideals dividing delta}
    Assume that $\pe$ divides either $D_A$ or $D_B$.
    Since $D_A$ and $D_B$ are fourth power free and the product $D_A D_B$ is in $(K^\times)^4$, both $D_A$ and $D_B$ are divided by $\pe$.
    For two elements $a,b\in \calO_K$, let $\gcd(a,b)$ be the largest ideal dividing both $a$ and $b$.
    Since we have the relation
    \[
        \gcd(D_A, D_B) | \gcd(A,B) | D^4 (3\theta^2 - 1) Z^8,
    \]
    $\pe$ divides $D$, $3\theta^2-1$, or $Z$.
    Since the valuations $v_{\pe}(D_A)$ and $v_{\pe}(D_B)$ are even by \cref{lem: elliptic curve} \cref{item: lem: x-T map},
    $\pe$ divides $D^4(3\theta^2 - 1)Z^8$ twice, and hence $\pe$ divides $D$ or $Z$.
    The equality $X = A + D^2 \theta Z^4 = D_A s_A^4 + D^2 \theta Z^4$ implies that $\pe$ divides $X$.
    Since $Z$ is coprime to $X$, $\pe$ does not divide $Z$, and hence $\pe$ divides $D$.
    
    \cref{item: lem: 2 does not divide deltaA deltaB}
    Assume the contrary.
    If $4$ does not divide $X$, then the $2$-adic valuations of both sides of \cref{eq:eliminatingdenominators} are distinct.
    Thus, $4$ divides $X$.
    Put $X' = X/4$. We have
    \[
        Y^4 = 64(X'^3 - (D/2)^4 X' Z^8 + (D/2)^6 Z^{12}).
    \]
    Again, comparing the $2$-adic valuations of both sides, we see that $4$ divides $Y$.
    Put $Y' = Y/4$. Noting that $D/2$ and $Z$ are odd integers, we have the following congruence relations
    \begin{align*}
        0 &\equiv 4Y' = (X')^3 - (D/2)X'Z^8 + (D/2)^6 Z^{12}\\
        &\equiv (X')^3 - X' + 1.
    \end{align*}
    However, $2$ does not divide the value $(X')^3 - X' + 1$ for all integers $X'$, which is a contradiction.

    \cref{item: lem: pe1 does not divide deltaA deltaB}
    Assume the contrary.
    As in \cref{item: lem: 2 does not divide deltaA deltaB}, the valuations $v_{\pe_1}(D_A)$ and $v_{\pe_1}(D_B)$ are even and positive.
    Since $D_A$ and $D_B$ are fourth power free, we have $v_{\pe_1}(D_A)= v_{\pe_1}(D_B) = 2$.
    By \cref{item: lem: prime ideals dividing delta}, $\pe_1$ divides $D$ and $X$.
    By the definition of $D$, the valuation $v_{\pe_1}(D)$ equals $1$ in this situation.
    Since $X$ and $Z$ are prime to each other, $Z$ is not divided by $\pe_1$.
    We have the equation
    \begin{equation} \label{eq:defofdeltaBsB}
        D_B s_B^4 = X^2 + D^2\theta XZ^4 + D^4(\theta^2 - 1)Z^8
    \end{equation}
    by the definition of $D_B$ and $s_B$.
    By dividing by $D^4Z^8$ and completing the square, we get
    \begin{equation} \label{eq:completingthesquare}
        D_B \left( \frac{s_B}{D Z^2}\right)^4 - \frac{3\theta^2 - 4}{4} = \left(\frac{X}{D^2 Z^4} + \frac{\theta}{2}\right)^2.
    \end{equation}
    If $v_{\pe_1}(s_B) > 0$, the $\pe_1$-adic valuation of both sides of \cref{eq:completingthesquare} are distinct.
    Thus, we obtain $v_{\pe_1}(s_B) = 0$.
    Moreover, since we have
    \begin{align*}
        v_{\pe_1}(\text{left hand side of \cref{eq:defofdeltaBsB}}) &= v_{\pe_1}(D_B) = 2,\\
        v_{\pe_1}(\text{right hand side of \cref{eq:defofdeltaBsB}}) &
        \begin{cases}
            = 0 & \text{if } v_{\pe_1}(X) = 0, \text{ and}\\
            \geq 4 & \text{if } v_{\pe_1}(X) \geq 2,
        \end{cases}
    \end{align*}
    we obtain $v_{\pe_1}(X) = 1$.
    On the other hand, by the definition of $D_A$ and $s_A$, we have
    \begin{equation} \label{eq:defofdeltaAsA}
        D_A s_A^4 = X- D^2 \theta Z^4.
    \end{equation}
    In our situation, we get
    \begin{align*}
        v_{\pe_1}(\text{left hand side of \cref{eq:defofdeltaAsA}})
        &\equiv 2\ \mod 4, \text{ and}\\
        v_{\pe_1}(\text{right hand side of \cref{eq:defofdeltaAsA}}) &= 1.
    \end{align*}
    This is a contradiction.
    
    \cref{item: lem: pe2 does not divide deltaA deltaB}
    Assume the contrary.
    As in the argument of the proof of \cref{item: lem: pe1 does not divide deltaA deltaB}, we have $v_{\pe_2}(D_A) = v_{\pe_2}(D_B) = 2$.
    By \cref{item: lem: pe1 does not divide deltaA deltaB}, we have $v_{\pe_1}(D_A) = v_{\pe_1}(D_B) = 0$.
    Since $X$ is a rational integer, and the ramification degree $e_{\pe_2/23}$ is $2$, we have
    \[
        v_{23}(X) = \frac{1}{2}v_{\pe_2}(X).
    \]
    By \cref{item: lem: prime ideals dividing delta}, they are positive integers.
    Thus, $v_{\pe_2}(X)$ is a positive even integer.
    If $v_{\pe_2}(X) \geq 3$ holds, we have
    \begin{align*}
        v_{\pe_2}(\text{left hand side of \cref{eq:defofdeltaBsB}}) &\equiv 2 \ \mod 4, \text{ and}\\
        v_{\pe_2}(\text{right hand side of \cref{eq:defofdeltaBsB}}) &= 4.
    \end{align*}
    This is a contradiction. Thus, the value $v_{\pe_2}(X)$ must be equal to $2$, and we get
    \[
        v_{\pe_1}(X) = v_{23}(X) = \frac{1}{2}v_{\pe_2}(X) = 1.
    \]
    Then, we obtain the equalities
    \begin{align*}
        v_{\pe_1}(\text{left hand side of }\cref{eq:completingthesquare}) &\equiv 0\ \mod 4, \text{ and}\\
        v_{\pe_1}(\text{right hand side of }\cref{eq:completingthesquare}) &= 2.
    \end{align*}
    This is a contradiction.

    \cref{item: lem: candidates of deltaA deltaB}
    By \cref{item: lem: prime ideals dividing delta}, only the prime ideal dividing $D$ can divide $D_A$ or $D_B$.
    Such prime ideals are $2\calO_K$, $\pe_1$, or $\pe_2$.
    \cref{item: lem: 2 does not divide deltaA deltaB}, \cref{item: lem: pe1 does not divide deltaA deltaB}, and \cref{item: lem: pe2 does not divide deltaA deltaB} implies that they do not divide $D_A$ and $D_B$.
    Hence, $D_A$ and $D_B$ are both units of $\calO_K$.
    The assertion follows from \cref{lem: elliptic curve} \cref{item: lem: x-T map} and \cref{lem: number field} \cref{item: lem: unit group}.
    \end{proof}

\begin{continuedproof}{Proof of \cref{thm: rational points of CD}}
Let $C_D^{\circ}$, $C_{D_A,D_B}$, and $C_{D_B}'$ be the curves defined as follows
\begin{align*}
    & C_D^{\circ} : y^4 = x^3 - D^4x + D^6\\
    & C_{D_A,D_B} :
    \begin{cases}
        D_As_A^4 = x - D^2\theta\\
        D_Bs_B^4 = x^2+ D^2\theta x + D^4(\theta^2 -1)\\
        D_AD_B = s^4
    \end{cases}\\
    & C_{D_B}' \colon y^2 = D_B x^4 - (3\theta^2 - 4)/4.
\end{align*}
Then, we have the following morphisms
\begin{center}
    \begin{tikzcd}
    	C_D & \ni (\frac{x}{D^2},\frac{s_A s_B s}{D^2}) & (x,y) \\
    	C_D^\circ & \ni (x, s_A s_B s) & (D^2x, D^2y) \\
    	C_{D_A,D_B} & \ni(x, s_A, s_B, s) \\
    	C_{D_B}' & \ni(\frac{s_B}{D}, \frac{x}{D^2}+ \frac{\theta}{2}) & (x,y) \\
    	\mathbb{P}^1 & \ni \frac{x}{D^2} & y-\frac{\theta}{2}
    	\arrow[maps to, from=1-3, to=2-3]
    	\arrow["\sim", from=2-1, to=1-1]
    	\arrow[maps to, from=2-2, to=1-2]
    	\arrow[from=3-1, to=2-1]
    	\arrow[from=3-1, to=4-1]
    	\arrow[maps to, from=3-2, to=2-2]
    	\arrow[maps to, from=3-2, to=4-2]
    	\arrow["\xi", from=4-1, to=5-1]
    	\arrow[maps to, from=4-3, to=5-3]
        \arrow[maps to, from=4-2, to=5-2]
    \end{tikzcd}
\end{center}
For all rational points $P\in C_D(\Q)$, there is a tuple $(D_A,D_B) = (1,1), (-\theta, -\theta^3), (\theta^2,\theta^2), (-\theta^3, -\theta)$ such that $P$ is an image of a $K$-rational point $(x,s_A,s_B,s) \in C_{D_A,D_B}(K)$ with $x\in \Q$.
Let $P'\in C_{D_B}'(K)$ be the image of $(x,s_A,s_B,s)$.
Then, the value $\xi(P')$ is a rational number.
Hence, we can solve our task by the following procedure using \textsc{Magma}.
\begin{enumerate}
    \item \label{item: thm: find a K rational point}
        Find a $K$-rational point $P_0\in C_{D_B}'(K)$ by using the command \texttt{RationalPoints}.
        Using the command \texttt{EllipticCurve} for $C_{D_B'}$ and $P_0$, we obtain an elliptic curve $E_{D_B}$ and an isomorphism $\varphi_{D_B}\colon C_{D_B}' \longrightarrow E_{D_B}$ which sends $P_0$ to the unit element $O$.
    \item \label{item: thm: finite index subgroup}
        Get an abstract abelian group $G$ and an embedding $m\colon G \hookrightarrow E_{D_B}(K)$ such that the index of its image in $E_{D_B}(K)$ is finite and coprime to $6$.
    \item \label{item: thm: Elliptic Chabauty}
        Use the command \texttt{Chabauty} for $m$ and $\xi\circ \varphi_{D_B}^{-1} \colon E_{D_B} \longrightarrow \bbP^1$ with index bound $6$.
        It tells us the set
        \[
            S = \left\{ g \in G \ \middle|\ \xi \circ \varphi_{D_B}^{-1}\circ m(g) \in \bbP^1(\Q)\right\}
        \]
        and an integer $I$ whose prime divisors are only $2$ or $3$.
        The output $I$ means that if the index $[E_{D_B}(K): m(G)]$ is coprime to $I$, the set $S$ equals the set of all points of $E_{D_B}(K)$ whose image under $\xi\circ \varphi_{D_B}^{-1}$ is in $\bbP^1(\Q)$.
        Since we constructed $G$ as its index is coprime to $6$ in \cref{item: thm: finite index subgroup}, this is the case if the computation stops.
\end{enumerate}
We write the computational code in \cref{appendix: sec: code} for $D_B = 1$. One can check the other cases by changing $n$ to $ 1$, $2$, and $3$ in the code.
By these computations, for $D_B = (-\theta)^n$ with $0\leq n \leq 3$, we obtain that the set
\[
    S' = \left\{ Q \in C_{D_B}'(K)\ |\ \xi(Q) \in \bbP^1(\Q) \right\}
\]
is as follows
\[
    S' =
    \begin{cases}
        \{ [1 , -1 , 0], [\theta , (\theta - 2)/2 , 1], [1 , 1 , 0], [-\theta , (\theta - 2)/2 , 1]\} & n= 0\\
        \{[-\theta^2 + 1 , (\theta + 2)/2 , 1],
        [\theta^2 - 1 , (\theta + 2)/2 , 1)\} & n= 1\\
        \emptyset & n= 2\\
        \{[\theta^2 - 1 , \theta/2 , 1],
        [-\theta^2 + 1 , \theta/2 , 1]\} & n= 3\\
    \end{cases}
\]
where we regard $C_{D_B}'$ is embedded into the weighted projective space $\bbP(1,2,1)$.
The image of these points under $\xi$ is either $0$ or $\pm 1$.
This means that the original point $P=(x,y) \in C_D(\Q)$ satisfies $x\in \{0, \pm 1\}$.
Hence, we obtain the equality
\[
    x^3 - x + 1 = 1.
\]
Consequently, the value $D$ must be $\pm 1$, and our assertion holds.
\end{continuedproof}

%%%%%%%%%%%%%%%%%%%%%%%%%%%%%%%%%%%%%%%%%%%%%%%%%%%%%%%%%%%%%%

\subsection{Completion of proof}
\label{subsec: proof of main theorem for d=2}

%%%%%%%%%%%%%%%%%%%%%%%%%%%%%%%%%%%%%%%%%%%%%%%%%%%%%%%%%%%%%%

Now, we are ready to prove our main theorem for $d=2$.
\begin{proof}[Proof of \cref{thm: main}\cref{item: thm: d = 2}]
    Let notation be as in \cref{subsec: geometry}.
    By \cref{subsec: reduction to CD}, for any rational point $x\in X_4(\Q)$, the point $\pi(x)$ equals $O$, $3Q_0$, or $\varphi \circ \psi_D(u,t)$ for some $D = (-1)^{\varepsilon_0} 2^{\varepsilon_1} 23^{\varepsilon_2}$, $(u,t) \in C_D(\Q)$, and $\varphi = [2]$ or $T_{3Q_0}\circ [2]$.
    Assume that it is in the latter case.
    By \cref{thm: rational points of CD}, the value $D$ must be $\pm 1$, and $(u,t)$ is in $\{(0,\pm 1), (\pm 1, \pm 1)\}$.
    The image of each $(u,t)$ in $E(\Q)$ is as follow:

\begin{center}
    \begin{tikzcd}
    	E & E & C_1 & E & E & C_{-1} \\
    	-2Q_0 & (1,1)=-Q_0 & (1,\pm1) & 2Q_0 & (1,-1)=Q_0 & (1,\pm1) \\
    	Q_0 &&& 5Q_0 \\
    	-4Q_0 & (-1,1)=-2Q_0 & (-1,\pm1) & 4Q_0 & (-1,-1)=2Q_0 & (-1,\pm1) \\
    	-Q_0 &&& 7Q_0 \\
    	6Q_0 & (0,1) = 3Q_0 & (0,\pm1) & -6Q_0 & (0,-1) = -3Q_0 & (0,\pm1) \\
    	9Q_0 &&& -3Q_0
    	\arrow["\varphi"', from=1-2, to=1-1]
    	\arrow["{\psi_1}"', from=1-3, to=1-2]
    	\arrow["\varphi"', from=1-5, to=1-4]
    	\arrow["{\psi_{-1}}"', from=1-6, to=1-5]
    	\arrow["{[2]}"', maps to, from=2-2, to=2-1]
    	\arrow["{T_{3Q_0} \circ [2]}", maps to, from=2-2, to=3-1]
    	\arrow[maps to, from=2-3, to=2-2]
    	\arrow["{[2]}"', maps to, from=2-5, to=2-4]
    	\arrow["{T_{3Q_0} \circ [2]}", maps to, from=2-5, to=3-4]
    	\arrow[maps to, from=2-6, to=2-5]
    	\arrow["{[2]}"', maps to, from=4-2, to=4-1]
    	\arrow["{T_{3Q_0} \circ [2]}", maps to, from=4-2, to=5-1]
    	\arrow[maps to, from=4-3, to=4-2]
    	\arrow["{[2]}"', maps to, from=4-5, to=4-4]
    	\arrow["{T_{3Q_0} \circ [2]}", maps to, from=4-5, to=5-4]
    	\arrow[maps to, from=4-6, to=4-5]
    	\arrow["{[2]}"', maps to, from=6-2, to=6-1]
    	\arrow["{T_{3Q_0} \circ [2]}", maps to, from=6-2, to=7-1]
    	\arrow[maps to, from=6-3, to=6-2]
    	\arrow["{[2]}"', maps to, from=6-5, to=6-4]
    	\arrow["{T_{3Q_0} \circ [2]}", maps to, from=6-5, to=7-4]
    	\arrow[maps to, from=6-6, to=6-5]
    \end{tikzcd}
\end{center}

Hence, combining them with the first two cases $\pi(x)= O$ or $3Q_0$, the point $\pi(x)$ is in the set
\[
    S = \{ O, \pm Q_0, \pm 2Q_0, \pm 3Q_0, \pm 4Q_0, 5Q_0, \pm 6Q_0, 7Q_0, 9Q_0\}.
\]
By direct calculation, we conclude that
\begin{equation*}
    X_4(\Q) = \pi^{-1}(S) \cap X_4(\Q) = \{P_i \ |\ 1\leq i \leq 10\},
\end{equation*}
where $P_i$'s are the points listed in \cref{lem: elliptic curve} \cref{item: lem: known points}.
\end{proof}

% --------------------------------------------------------------------------

\section{Proof of main results for $d \geq 3$} \label{sec: d geq 3}

% --------------------------------------------------------------------------
We prove \cref{thm: main} \cref{item: thm: d geq 3} by reducing the problem to the following theorem.

\begin{thm}[{\cite[Main Theorem 2.]{DM97}}]\label{thm: Darmon-Merel}
    The equation
    \begin{equation} \label{eq:Darmon--Merel}
        x^n + y^n = z^2
    \end{equation}
    has no non-trivial primitive solution when $n\geq 4$.
\end{thm}

Note that a solution $(x,y,z)$ of the equations in \cref{thm: Darmon-Merel} is said to be primitive if $x,y,z$ are integers and $\gcd(x,y,z) = 1$ holds, and to be trivial if $xyz = 0$ or $\pm 1$.
Since the Shimura--Taniyama Conjecture, the modularity of every elliptic curve over $\Q$, is proved by Wiles, Taylor, Breuil, Conrad, and Diamond in \cite{Wiles95}, \cite{TW95}, \cite{CDT99}, and finally \cite{BCDT99}, the conclusion of \cref{thm: Darmon-Merel} is now unconditionally valid.

For the case $d=3$ or $4$, we use the following proposition on the Mordell--Weil group of specific elliptic curves.
\begin{prop}\label{prop: elliptic curves y2 = x3 pm 1}
    \begin{parts}
        \item \label{item: prop: elliptic curve for d = 3}
            The Mordell--Weil group of the elliptic curve defined by
            \[
                y^2 = x^3 - 1.
            \]
            is isomorphic to $\Z/2\Z$, and consists of $[0, 1, 0]$ and $[1, 0, 1]$ in the homogeneous coordinate.
        \item \label{item: prop: elliptic curve for d = 4}
            The Mordell--Weil group of the elliptic curve defined by
            \[
                y^2 = x^3 + 1.
            \]
            is isomorphic to $\Z/6\Z$, and consists of the points
            \[
                [0 , 1 , 0], [2 , \pm 3 , 1],[0 , \pm 1 , 1],\text{ and } [-1 , 0 , 1]
            \]
            in the homogeneous coordinate.
    \end{parts}
\end{prop}
\begin{proof}
    See the entry in the LMFDB \cite[Elliptic Curve 144.a3 and 36.a4]{lmfdb}.
    One can prove the assertion by using elementary number theory or \textsc{Magma}.
\end{proof}

\begin{proof}[Proof of \cref{thm: main}\cref{item: thm: d geq 3}]
    Suppose $c \neq 0, -1$.
    Assume that we have $z_2 \in f_{d,c}^{-2}(0)(\Q)$, and put $z_1= f_{d,c}(z_2)$.
    Then, we have
    \[
        c = -z_1^d \neq 0
    \]
    by $f_{d,c}(z_1) = 0$.
    By the definition of $z_1$, we have
    \begin{equation} \label{eq:curvefordgeq3}
        \left(\frac{z_2}{z_1}\right)^d = \left(\frac{1}{z_1}\right)^{d-1} + 1.
    \end{equation}
    
    If $d=3$, the point $(z_2/z_1, 1/z_1)$ is a rational point of the elliptic curve defined by
    \[
        y^2 = x^3 - 1.
    \]
    Thus, the point $(z_2/z_1, 1/z_1)$ equals $[0, 1, 0]$ or $[1, 0, 1]$ by \cref{prop: elliptic curves y2 = x3 pm 1}\cref{item: prop: elliptic curve for d = 3}.
    However, this is impossible since we are assuming $-z_1^d = c \neq 0$.

    If $d = 4$, the point $(1/z_1, (z_2/z_1)^2)$ is a rational point of the elliptic curve defined by
    \[
        y^2 = x^3 + 1.
    \]
    Thus, the point $(1/z_1, (z_2/z_1)^2)$ equals either
    \[
        [0 , 1 , 0],\ [2 , \pm 3 , 1],\ [0 , \pm 1 , 1],\text{ or } [-1 , 0 , 1]
    \]
    by \cref{prop: elliptic curves y2 = x3 pm 1}\cref{item: prop: elliptic curve for d = 4}.
    However, this is impossible since $\pm 3$ are not squares of rationals, and we are assuming $-z_1^d = c \neq 0$.
    
    Let $d \geq 5$.
    By comparing the valuations of both sides of \cref{eq:curvefordgeq3} over all prime numbers, we can see that there are non-zero integers $A, B, C$ with $(A, B)= (B, C) = 1$ such that
    \[
        \frac{z_2}{z_1} = \frac{A}{B^{d-1}}, \quad \frac{1}{z_1} = \frac{C}{B^d}.
    \]
    These $A, B, C$ satisfies the equality
    \[
        A^d = C^{d-1} + B^{d(d-1)}.
    \]
    If $d \geq 5$ is even, the point $(C,B^d, A^{d/2})$ is a non-trivial primitive solution of \cref{eq:Darmon--Merel} for $n=d-1$.
    If $d \geq 5$ is odd, the point $(A, -B^{d-1},C^{(d-1)/2})$ is a non-trivial primitive solution of \cref{eq:Darmon--Merel} for $n=d$.
    They are contradictions.
\end{proof}

\begin{proof}[Proof of \cref{cor: number of rational preimages}]
    Suppose that $d \geq 3$ is even.
    If $c = 0$, rational iterated preimages of $0$ are only $0$.
    If $c = -1$, we have
    \begin{align*}
        f_{d,-1}^{-1}(0)(\Q) &= \{1,-1\},\\
        f_{d,-1}^{-1}(1)(\Q) &= \emptyset, \text{ and}\\
        f_{d,-1}^{-1}(-1)(\Q) &= \{0\}.
    \end{align*}    
    If $c = -r^{d}$ for some rational number $r \neq 0, \pm 1$, we have
    \[
        f_{d,c}^{-1}(0)(\Q) = \{ \pm r\}.
    \]
    By \cref{thm: main}\cref{item: thm: d geq 3}, there are no second rational preimages of $0$ under $f_{d,c}$.
    For the remaining case that $c$ is not of the form $-r^d$ with a rational number $r$, the set $f_{d,c}^{-1}(0)(\Q)$ is empty.
    The assertion for even $d$ follows from these observations.

    Suppose that $d\geq 3$ is odd.
    If $c = 0$, the rational iterated preimages of $0$ are only $0$.
    If $c = -r^d$ for some rational number $r\neq 0$, we have
    \[
        f_{d,c}^{-1}(0)(\Q) = \{r\}.
    \]
    By \cref{thm: main} \cref{item: thm: d geq 3}, there are no rational second iterated preimages of $0$ under $f_{d,c}$.
    If $c$ is not of the form $-r^d$ with a rational number $r$, the set $f_{d,c}^{-2}(0)(\Q)$ is empty.
    The assertion for odd $d$ follows from these observations.
\end{proof}

\begin{appendix}
\section{Magma code for the elliptic Chabauty method}
\label{appendix: sec: code}
We write the code of \textsc{Magma} to determine the $K$-rational points of $C_{D_B}$ whose image under $\xi$ is in $\bbP^1(\Q)$.
We used it in the proof of \cref{thm: rational points of CD}.

\begin{magmacode}[caption={Magma code for the elliptic Chabauty method}]
Q := Rationals();
Qt<t> := PolynomialRing(Q);
P1 := ProjectiveSpace(Q,1);
K<a> := ext<Q | t^3 - t + 1>;
Kx<x> := PolynomialRing(K);

// list up primes to saturate
primes_to_saturate := [2,3];

// set n = 0..3
n := 0;
C<X,Y,Z> := HyperellipticCurve((-a)^n*x^4 - 1/4*(3*a^2 - 4));
P := RationalPoints(C : Bound := 10)[1];
E, phi := EllipticCurve(C, P);

T, piT := TorsionSubgroup(E);
IT := Invariants(T);
MWR := MordellWeilRank(E);

// collecting candidates by 2-descent
Covers, maps := TwoDescent(E);
pts := [];
for i in [1..#Covers] do
    PtsCi := Points(Covers[i] : Bound := 1);
    if #PtsCi gt 0 then
        _, phiCi := AssociatedEllipticCurve(Covers[i] : E := E);
        for Q in PtsCi do
            Append(~pts, phiCi(Q));
        end for;
    end if;
end for;

// picking up independent elements
S := [];
if #pts gt 0 then
    _ := IsLinearlyIndependent(pts);
    S := ReducedBasis(pts);

    // saturation for each p
    for p in primes_to_saturate do
        S := Saturation(S, p : TorsionFree := true);
    end for;
end if;

// found free rank and the flag
r_found := #S;
flag_fullrank := (r_found eq MWR);

// construction of A, imgs, phi
A := AbelianGroup([0 : i in [1..r_found]] cat IT);
imgs := S cat [ piT(g) : g in Generators(T) ];
m := hom< A -> E | a :-> &+[ Eltseq(a)[i]*imgs[i] : i in [1..#imgs] ]>;
if not flag_fullrank then;
    print "need more independence elements";
end if;

if flag_fullrank then;
    // definition of function Ecov
    xi := map<C->P1 | [Y - a/2*Z^2, Z^2]>;
    Ecov := phi^(-1)*xi;

    // run Chabauty method
    Chab, I := Chabauty(m, Ecov: IndexBound:=6);

    // wanted points in C_DB
    print "n = ", n;
    print "K-rational points of C_DB with rational image of xi";
    for c in Chab do;
        print (m*phi^(-1))(c);
    end for;
end if
\end{magmacode}

\end{appendix}

% --------------------------------------------------------------------------
%		References
% --------------------------------------------------------------------------
%\newpage
\bibliographystyle{alpha}
\bibliography{dynamic}

% --------------------------------------------------------------------------
\end{document}